\documentclass[psamsfonts,fceqn,leqno]{amsart}
\usepackage{mathrsfs,latexsym,amsfonts,amssymb,curves,epic}
\usepackage{color}

\newtheorem{theorem}{Theorem}[section]
\newtheorem{corollary}[theorem]{Corollary}
\newtheorem{proposition}[theorem]{Proposition}
\newtheorem{lemma}[theorem]{Lemma}
\newtheorem{question}[theorem]{Question}
\newtheorem{problem}[theorem]{Problem}
\theoremstyle{definition}
\newtheorem{definition}[theorem]{Definition}
\newtheorem{remark}[theorem]{Remark}
\newtheorem{example}[theorem]{Example}

\begin{document}
\title[The Transversality on locally pseudocompact groups]
{The Transversality on locally pseudocompact groups}

  \author{Fucai Lin}
  \address{(Fucai Lin): School of mathematics and statistics,
  Minnan Normal University, Zhangzhou 363000, P. R. China}
  \email{linfucai2008@aliyun.com; linfucai@mnnu.edu.cn}

  \author{Zhongbao Tang}
  \address{(Zhongbao Tang): School of mathematics and statistics,
  Minnan Normal University, Zhangzhou 363000, P. R. China}
  \email{tzbao84@163.com}

  \thanks{The authors are supported by the NSFC (No. 11571158), the Natural Science Foundation of Fujian Province (No. 2017J01405) of China, the Program for New Century Excellent Talents in Fujian Province University, the Institute of Meteorological Big Data-Digital Fujian and Fujian Key Laboratory of Data Science and Statistics.}


  \begin{abstract}
Two non-discrete Hausdorff group topologies $\tau, \delta$ on a group $G$ are called {\it transversal} if the least
upper bound $\tau\vee \delta$ of $\tau$ and $\delta$ is the discrete topology. In this paper, we discuss the existence of transversal group topologies on locally pseudocompact, locally precompact or locally compact groups. We prove that each locally pseudocompact, connected topological group satisfies CSP, which gives an affirmative answer to a problem posed by Dikranjan,
Tkachenko and Yaschenko in 2006. For a compact normal subgroup $K$ of a locally compact totally disconnected group $G$, if $G$ admits a transversal group topology then $G/K$ admits a transversal group topology, which give a partial answer again to a problem posed by Dikranjan, Tkachenko and Yaschenko in 2006. Moreover, we characterize some classes of locally compact groups that admit transversal group topologies.
  \end{abstract}

 \maketitle
\section{Introduction}
Two non-discrete Hausdorff group topologies $\tau, \delta$ on a group $G$ are called {\it transversal} if the least
upper bound $\tau\vee \delta$ of $\tau$ and $\delta$ is the discrete topology. Indeed, the transversality, which was introduced in \cite{TTWY2000}, is a ``half'' of the stronger concept of complementarity introduced by Birkhoff \cite{B1936}. The study of transversal group topologies on abelian groups was initiated in \cite{DTY2005} and \cite{ZP2001}, and the authors in \cite{DTY2005} proved that any precompact
topological group does not admit a transversal group topology. In \cite{DTY2006}, D. Dikranjan, M. Tkachenko and I. Yaschenko continued the study of transversal group topologies on non-abelian groups; in particular, they obtained the following important duality principle for transversal topologies.

\begin{theorem}\cite[Theorem 2.2]{DTY2006}
Let $\tau_{1}, \tau_{2}$ be transversal group topologies on an infinite group $G$. Then $$\triangle(G, \tau_{1})\leq in(G, \tau_{2})\ \mbox{and}\ \triangle(G, \tau_{2})\leq in(G, \tau_{1}),$$
where $\triangle(G)$ and $in(G)$ denote by the dispersion character and index of narrowness of $G$ respectively.
\end{theorem}

Further, D. Dikranjan, M. Tkachenko and I. Yaschenko in \cite{DTY2006} also discussed the impact of weak commutativity and weak $\omega$-bounded on transversality. The following three problems were posed in \cite{DTY2006}.

\begin{problem}\cite[Problem 6.1]{DTY2006}\label{p1}
Let $G$ be a topological group admitting a transversal group topology and $K$ be a compact normal subgroup of $G$. Does $G/K$ admit a transversal group topology?
\end{problem}

{\bf Note} It follows from \cite[Example 5.4]{DTY2005} that even if $G/K$ admits a
transversal group topology, for a compact normal subgroup $K$ of
$G$, the group $G$ can fail admitting a transversal group topology.

\begin{problem}\cite[Problem 6.2]{DTY2006}\label{p2}
Characterize the locally compact groups that admit a transversal group topology.
\end{problem}

A Tychonoff space is said to be {\it pseudocompact} if its image under any continuous function to $\mathbb{R}$ is bounded. In \cite{DTY2006}, D. Dikranjan, M. Tkachenko and I. Yaschenko proved that if a locally compact and connected topological group $G$ admits a transversal group topology then either the central subgroup $Z(G)$ is infinite discrete or $Z(G)$ admits a transversal group topology. Then they posed the following problem.

\begin{problem}\cite[Problem 6.3]{DTY2006}\label{p3}
If a locally pseudocompact and connected topological group $G$ admits a transversal group topology, then either is the central subgroup $Z(G)$ infinite discrete or does $Z(G)$ admit a transversal group topology?
\end{problem}

In this paper, we shall give an affirmative answer to Problem~\ref{p3} and some partial answers to Problems~\ref{p1} and~\ref{p2} respectively. The paper is organized as follows.

In Section 2, we introduce the necessary notation and terminology which are
used in the rest of the paper. In Section 3, we investigate the
locally pseudocompact connected topological groups. We prove that each locally pseudocompact, connected topological group satisfies CSP, which gives an affirmative answer to \cite[Problem 6.3]{DTY2006}. Moreover, we discuss the transversality in some classes of locally pseudocompact groups and extend some results in \cite{DTY2006}. In Section 4, we study the transversality of locally precompact groups. In particular, we prove that for a locally precompact connected abelian group $G$ and any compact subgroup $K$ of $G$, if $G$ admits a transversal group topology then $G/K$ also admits a transversal group topology, which gives a partial answer to Problem~\ref{p1}. In Section 5, we mainly discuss central topological groups. First, we give a characterization of central topological groups that admit transversal group topologies. Moreover, we prove that for a connected central topological group $G$ and any compact subgroup $K$ of $G$, if $G$ admits a transversal group topology then $G/K$ also admits a transversal group topology, which also gives a partial answer to Problem~\ref{p1}.

 \maketitle
\section{Notation and terminology}
We denote by $\mathbb{N}$ and $\mathbb{R}$ the set of all positive
  integers and the reals, respectively, by $\mathbb{P}$ the set of primes. For undefined
  notation and terminology, the reader may refer to \cite{AT2008} and
  \cite{E1989}.

For a group $G$, we denote by $e$ the neutral element, by $Z(G)$ the center of $G$, by $C(G)$ the connected component of the neutral element. The {\it socle} of $G$ is defined by $Soc(G)=\bigoplus_{p\in\mathbb{P}}G[p]$, where $G[p]=\{x\in G: px=0\}$ for each $p\in\mathbb{P}$. The group $G$ is called {\it divisible} if for every $g\in G$ and $n\in\mathbb{N}$ the equation $x^{n}=g$ has a solution in $G$. Abelian groups will be written additively. If $G$ is a topological group, then its index of {\it narrowness} $in(G)$ is the minimal infinite cardinal number $\kappa$ such that $G$ can be covered by at most $\kappa$ translates of each neighborhood of the neutral element. Topological group $G$ satisfying $in(G)\leq\omega$ is called {\it $\omega$-narrow}. The topological group $G$ is called {\it precompact} if $G$ can be covered by finitely many translates of each neighborhood of the neutral element. The concept of precompactness admits a natural extension to subsets of topological groups as follows. A subset $B$ of a topological group $G$ is said to be {\it precompact} in $G$, if for each neighborhood $U$ of the identity in $G$, one can find a finite set $F\subset G$ such that
$B\subset FU $ and $B\subset UF$. It is well known that all separable groups and $\sigma$-compact groups are $\omega$-narrow and each pseudocompact topological group is precompact, see \cite{AT2008}.

A topological group $G$ is said to be {\it locally compact} (resp., {\it locally pseudocompact, locally precompact}) if there exists a neighborhood $U$ of the neutral element $e$ such that $U$ is compact (resp., pseudocompact, precompact).

\begin{definition}\cite{DTY2006}
An infinite group $G$ is said to be {\it weakly abelian} if for each countable set $U\subset G$, there exists a finite set $F\subset G$ such that $\bigcap_{x\in F}U^{x}\subset Z(G)$, where $U^{x}$ denotes the set $x^{-1}Ux$.
\end{definition}

The class of topological groups that admit transversal group topologies is denoted by $\mbox{{\bf Trans}}$. Recall that a topological group $G$ satisfies the {\it central subgroup paradigm} \cite{DTY2006}, briefly, $G\in CSP$ if
$$G\in \mbox{{\bf Trans}}\Rightarrow Z(G)\in\mbox{{\bf Trans}}\ \mbox{or}\ Z(G)\ \mbox{is infinite discrete}.$$

For a topological group $G$ and each $g\in G$, let $c_{G}(g)=\{x\in G: xg=gx\}$ be the {\it centralizer} \cite{DTY2006} of $g$ in $G$. Obviously, each $c_{G}(g)$ is a closed subgroup. Consider the following condition on $G$:\\

\maketitle
(NC) For no $g\in G\setminus Z(G)$, the subgroup $c_{G}(g)$ of $G$ is open.

 \maketitle
\section{locally pseudocompact connected topological groups}
In this section, we mainly discuss the transversality in locally pseudocompact connected topological groups, and give an affirmative answer to Problem~\ref{p3}. First, we recall a principal result about topological groups which we need in the
sequel.

\begin{theorem}\cite[Weil]{W1937}\label{th8}
Every locally precompact group $G$ embeds as a dense
topological subgroup to a locally compact group $\widetilde{G}$.
\end{theorem}

The following theorem gives an affirmative answer to Problem~\ref{p3}.

\begin{theorem}\label{t0}
Each locally pseudocompact, connected topological group satisfies CSP.
\end{theorem}

\begin{proof}
Let $G$ be a locally pseudocompact, connected topological group. Then, from Theorem~\ref{th8}, it follows that $G$ can embed as a dense topological subgroup of a locally compact group $\widetilde{G}$. Hence $\widetilde{G}$ must be connected by the connectedness of $G$ and its denseness in $\widetilde{G}$. Then $\widetilde{G}$ is a locally compact, connected topological group, hence is $\sigma$-compact and thus $\omega$-narrow. Therefore, it follows from \cite[Proposition 5.1.1]{AT2008} that $G$ is also $\omega$-narrow. By the connectedness of $G$, we see that $G$ also satisfies (NC). Since $G$ is a locally pseudocompact, $G$ is locally Baire and, hence, it is Baire. From \cite[Lemma 2.12]{DTY2006}, we know that a Baire topological group satisfying (NC) is weakly abelian. Then $G$ is a weak abelian, $\omega$-narrow topological group, hence it follows from \cite[Corollary 2.17]{DTY2006} that $G$ satisfies CSP.
\end{proof}

The following corollary can be easily obtained from Theorem~\ref{t0} and \cite[Theorem 2.2]{DTY2006}.

\begin{corollary}\label{c1}
A locally pseudocompact, connected topological group with precompact center does not admit a transversal group topology.
\end{corollary}

A Hausdorff topological group $(G, \tau)$ is {\it minimal} if $\tau$ is a minimal element of the partially ordered set of Hausdorff group topologies on the group $G$  \cite{S1971}; we say that $G$ is {\it transversable} if $G$ admits a transversal group topology.

\begin{corollary}\label{c2}
Each locally pseudocompact, connected minimal topological group $G$ is not transversable.
\end{corollary}

\begin{proof}
By \cite[Theorem 2.2]{D1998}, $Z(G)$ is minimal. Hence $Z(G)$ is precompact by \cite{PS1984}, then $Z(G)$ is not transversable by \cite[Theorem 3.13]{DTY2005}. Now Corollary~\ref{c1} applies.
\end{proof}

We conjecture that the answer to the following question is affirmative.

\begin{question}\label{q1}
Let $G$ be a connected locally pseudocompact group. If $G$ is not transversable, is $Z(G)$ precompact?
\end{question}

The following theorem gives a partial answer to Question~\ref{q1}.

\begin{theorem}\label{th4}
Let $G$ be a connected locally pseudocompact group and $Z(G)$ be divisible or finitely generated. Then $G$ is not transversable if and only if $Z(G)$ is precompact. If, an addition, $G$ is locally compact, then $G$ is not transversable if and only if $Z(G)$ is compact.
\end{theorem}

\begin{proof}
Sufficiency. Assume $Z(G)$ is precompact. By Theorem~\ref{t0}, $G$ is not transversable by \cite[Theorem 3.13]{DTY2005}.

Necessity. In this part of our argument, we do not use the assumption
that $G$ is connected. Suppose that $G$ is not transversable. We prove that $Z(G)$ is precompact. Suppose not, if $Z(G)$ is infinite discrete, then $G$ is transversable by \cite[Corollary 3.9]{DTY2005}, which is a contradiction. Hence we can assume that $Z(G)$ is non-discrete. Since $Z(G)$ is divisible or finitely generated, it follows from  \cite[Lemma 4.8 and Theorem 4.6]{DTY2005} that $Z(G)$ is transversable. Applying \cite[Corollary 3.9]{DTY2005} again, we see that $G$ is transversable, which is a contradiction. Therefore, $Z(G)$ is precompact.
\end{proof}

By the necessity part of the proof of Theorem~\ref{th4}, we have the following corollary.

\begin{corollary}\label{c3}
Let $G$ be a locally pseudocompact group such that $Z(G)$ is divisible or finitely generated. If $Z(G)$ is non-precompact, then $G$ is transversable.
\end{corollary}

By Corollary~\ref{c3}, we also have the following result.

\begin{corollary}
Let $G$ be a locally pseudocompact, non-pseudocompact, divisible or finitely generated abelian group. Then $G$ is transversable.
\end{corollary}

The following theorem is a generalization of \cite[Corollary 3.1]{DTY2006}. First, we need a lemma.

\begin{lemma}\cite[Proposition 1.3]{EKD1980}\label{llllll}
Each pseudocompact space $X$ without isolated points has cardinality at least $\mathfrak{c}$.
\end{lemma}

\begin{theorem}
If $(G, \tau)$ is a locally pseudocompact group admitting a locally pseudocompact transversal group topology, then $in(G)\geq\mathfrak{c}$ and, consequently, $|G/C(G)|\geq\mathfrak{c}$.
\end{theorem}

\begin{proof}
Assume that $\sigma$ is a locally pseudocompact group topology on $G$ transversal to $\tau$. By definition, $\sigma$ is non-discrete. We claim that each non-empty open subset of $\sigma$ has size at least $\mathfrak{c}$ by Lemma~\ref{llllll}. Indeed, if $U$ is a nonempty open set in $(G, \sigma)$, we choose an arbitrary $x\in U$ and an open neighborhood $V$ of the identity $e$ in $G$ such that $\overline{V}$ is pseudocompact and $x\overline{V}\subset U$. Then $|U|\geq |x\overline{V}|=|\overline{V}|\geq\mathfrak{c}$ because the set $\overline{V}$ is pseudocompact and has no isolated points. Therefore, $in(G)\geq\mathfrak{c}$ by \cite[Theorem 2.2]{DTY2006}.

Let us prove that $|G/C(G)|\geq\mathfrak{c}$. Since $in(G)=in(C(G))\cdot in(G/C(G))$ by \cite[Assertion 2.1]{DTY2006}, it suffices to prove $in(C(G))\leq\omega$. Since the group $(G, \tau)$ is locally pseudocompact, it follows from Theorem~\ref{th8} that $(G, \tau)$ can embed as a dense topological subgroup to a locally compact group $\widetilde{G}$. The connectedness of $C(G)$ implies that $\overline{C(G)}^{\widetilde{G}}$ is a connected subgroup of $\widetilde{G}$, hence $C(G)\subset\overline{C(G)}^{\widetilde{G}}\subset C(\widetilde{G})$. Since $C(\widetilde{G})$ is closed in $\widetilde{G}$, it follows that $C(\widetilde{G})$ is a locally compact connected subgroup. Hence $C(\widetilde{G})$ is $\sigma$-compact, which in turn implies that $in(C(\widetilde{G}))\leq\omega$. Then $in(C(G))\leq\omega$ by \cite[Proposition 5.1.1]{AT2008}.
\end{proof}

\begin{corollary}\cite[Corollary 3.1]{DTY2006}
If $(G, \tau)$ is a locally compact group admitting a Locally compact transversal group topology, then $in(G)\geq\mathfrak{c}$ and, consequently, $|G/C(G)|\geq\mathfrak{c}$.
\end{corollary}

The inequality $in(G)\geq\mathfrak{c}$ in Theorem~\ref{t0} implies that the
locally pseudocompact group G is neither connected nor separable. Therefore, no locally pseudocompact group that is either connected or separable can have a locally compact transversal group topology.

 \maketitle
\section{locally precompact topological groups}
In this section, we mainly discuss the transversality in locally precompact topological groups, and give some partial answers to Problems~\ref{p1} and ~\ref{p2}. The following theorem proved in \cite{DTY2005} gives a characterization
of locally precompact abelian groups admitting transversal group topologies.

\begin{theorem}\label{th6}\cite[Theorem 5.18]{DTY2005}
Let $G$ be a locally precompact non-discrete abelian group. Then $G$ does not admit a transversal group topology if and only if there exists $n\in\mathbb{N}$ such that $nG+Soc(G)$ is precompact, if and only if there exists $n\in\mathbb{N}$ such that $nG$ and $Soc(G)$ are precompact.
\end{theorem}

By Theorem~\ref{th6}, we have the following two propositions, which also give some partial answers to Problem~\ref{p2}. The infimum $\mathcal{M}_{G}$ of all maximal non-discrete group topologies on $G$ is call {\it the submaximal group topology of $G$}. First, we need a lemma.

\begin{lemma}\label{ll99}
Let $H$ be an abelian topological group and $G$ be a subgroup of $H$. If $G$ is open in $(H, \mathcal{M}_{H})$, then $\mathcal{M}_{G}=\mathcal{M}_{H}|_{G}$. 
\end{lemma}

\begin{proof}
From \cite[Corollary 3.3]{DTY2005}, we have $\mathcal{M}_{H}|_{G}\leq \mathcal{M}_{G}$. In order to see $\mathcal{M}_{G}\leq\mathcal{M}_{H}|_{G}$, it suffices to prove that $\tau|_{G}$ is a maximal group topology on $G$ for any non-discrete maximal group topology $\tau$ on $H$. Let $\tau$ be a non-discrete maximal group topology on $H$. Since $G$ is open in $(H, \mathcal{M}_{H})$, we see that $G\in\tau$, hence $\tau|_{G}$ is not discrete. We claim that $\tau|_{G}$ is a maximal group topology on $G$. Suppose not, there exists a non-discrete group topology $\sigma$ on $G$ which is strict finer than $\tau|_{G}$. Then we can can extend the group topology $\sigma$ from $G$ to $H$. Indeed, declare all $\sigma$-open subsets of $G$ open in $H$ to get a group topology $\eta$ on $H$. Then $\eta|_{G}=\sigma$ and $\eta$ is strict finer than $\tau$ since $G\in\tau$, which is a contradiction with the maximality of $\tau$ on $H$.
\end{proof}

\begin{proposition}\label{c4}
Let $G$ be a uncountable non-discrete locally precompact abelian group such that each precompact subgroup of $G$ has the Suslin property as a subspace in $(G, \mathcal{M}_{G})$. Then $G$ admits a transversal group topology.
\end{proposition}

\begin{proof}
Since $G$ is non-discrete, it is obvious that $(G, \mathcal{M}_{G})$ is non-discrete. Assume that $G$ does not admit a transversal group topology. By Theorem~\ref{th6}, there exists $n\in\mathbb{N}$ such that $H=nG+Soc(G)$ is precompact. Then $H$ has the Suslin property in $(G, \mathcal{M}_{G})$ by the assumption. Moreover, $H\in\mathcal{M}_{G}$ by \cite[Corollary 4.3]{DTY2005}, thus $H$ is open in $(G, \mathcal{M}_{G})$. It follows from Lemma~\ref{ll99} that $\mathcal{M}_{G}|_{H}=\mathcal{M}_{H}$, then $(H, \mathcal{M}_{H})$ has the Suslin property. Therefore, it follows from \cite[Corollary 3.12]{APTTW1998} that $H$ is countable. However, it is obvious that $|H|=|G|$ is uncountable, which is a contradiction. Hence $G$ admits a transversal group topology.
\end{proof}

Since a locally compact precompact subgroup is compact, we have the following proposition.

\begin{proposition}\label{cccc}
Let $G$ be a uncountable non-discrete locally compact abelian group such that each compact subgroup of $G$ has the Suslin property as a subspace in $(G, \mathcal{M}_{G})$. Then $G$ admits a transversal group topology.
\end{proposition}

\begin{proof}
Assume that $G$ does not admit a transversal group topology. By Theorem~\ref{th6}, there exists $n\in\mathbb{N}$ such that $nG+Soc(G)$ is precompact, then $\overline{nG+Soc(G)}^{G}$ is compact since it is a locally compact precompact subgroup. Let $H=\overline{nG+Soc(G)}^{G}$. Since $nG+Soc(G)$ is open in $(G, \mathcal{M}_{G})$, it follows that $H$ is open in $(G, \mathcal{M}_{G})$. By Lemma~\ref{ll99} and the assumption, we can conclude that $\mathcal{M}_{H}$ has the Suslin property. By a similar proof of Proposition~\ref{c4}, we can obtain a contradiction. Hence $G$ admits a transversal group topology.
\end{proof}

\begin{remark}
The condition ``uncountable'' is important in Proposition~\ref{cccc}. Indeed, the compact topological group $G=\mathbb{Z}(4)^{\omega}$ is separable metrizable, and each compact subgroup of $G$ has the Suslin property in $(G, \mathcal{M}_{G})$ since it is countable; however, $G$ does not admit a transversal group topology since it is compact.
\end{remark}

Finally, we give an answer to Problem~\ref{p1} in the class of locally precompact abelian groups. First, we need some lemmas.

\begin{lemma}\label{l2}
If a topological group $H$ contains a dense locally precompact abelian subgroup $G$, then $H$ is locally precompact and abelian .
\end{lemma}

\begin{proof}
Clearly, $H$ is abelian. Moreover, it follows from \cite[Lemma 3.7.5]{AT2008} that $H$ locally precompact.
\end{proof}

\begin{lemma}\label{l1}
Let a topological group $H$ contain a dense locally precompact abelian subgroup $G$. If $G$ satisfies one of the following conditions, then $G\in\mbox{{\bf Trans}}$ if and only if $H\in\mbox{{\bf Trans}}$.

\smallskip
(i) $G$ is connected;

\smallskip
(ii) The subgroup Soc($G$) is dense in Soc($H$).
\end{lemma}

\begin{proof}
By Lemma~\ref{l2}, we see that $H$ is abelian and locally precompact. Clearly, $H\in\mbox{{\bf Trans}}$ if $G\in\mbox{{\bf Trans}}$. Hence it suffices to consider the sufficiency. Assume that $H\in\mbox{{\bf Trans}}$ and $G\not\in\mbox{{\bf Trans}}$.

(i) Assume that $G$ is connected. Next, we will obtain a contradiction. Indeed, it follows from Theorem~\ref{th8} that $H$ can embed as a dense topological subgroup of a connected locally compact group $\widetilde{H}$, then $G$ is also a dense topological subgroup of $\widetilde{H}$. Clearly, $\widetilde{H}\in\mbox{{\bf Trans}}$. As every locally compact abelian connected group, $\widetilde{H}$ is a direct sum of a vector group $\mathbb{R}^{n}$ and a compact group $H_{0}$. Let us prove that $n=0$, which implies a contradiction since $\widetilde{H}$ is a compact group and $\widetilde{H}\in\mbox{{\bf Trans}}$. Since $G$ is dense in $\widetilde{H}$, it suffice to prove that $G\subset \{0\}\times H_{0}$. Indeed, assume $x=(r, g)\in\widetilde{H}\cap G\subset\mathbb{R}^{n}\times H_{0}$, where $r\neq 0$. Then $x$ generates a discrete subgroup $\langle x\rangle$ of $\widetilde{H}$ (as it projects onto the discrete subgroup $\langle r\rangle$ of $\mathbb{R}^{n}$). Then $G\in\mbox{{\bf Trans}}$ by \cite[Corollary 3.5]{DTY2005}, which is a contradiction. Therefore, $\widetilde{H}=H_{0}$ is compact. Hence $G\in\mbox{{\bf Trans}}$.

(ii) Assume the subgroup Soc($G$) is dense in Soc($H$). Then it follows from Theorem~\ref{th6} that there exists $n\in\mathbb{N}$ such that $nG$ and $Soc(G)$ are precompact. Since $G$ and Soc($G$) are dense in $H$ and Soc($H$) respectively, it is easy to see that $nH\subset\overline{nG}^{H}$ and $Soc(H)\subset\overline{Soc(G)}^{H}$. Then it follows from \cite[Lemma 3.7.5]{AT2008} that $nH$ and $Soc(H)$ are all precompact subgroup in $H$, hence $nH+Soc(H)$ is precompact subgroup in $H$. Therefore, $H$ does not admit a transversal group topology by Theorem~\ref{th6}, which is a contradiction. Therefore, $G\in\mbox{{\bf Trans}}$.
\end{proof}

\begin{lemma}\label{lll0}
Let $G$ be a subgroup of $H$ such that $Soc(G)$ is dense in $Soc(H)$. If $K$ is a divisible subgroup of $G$, then $Soc(G/K)$ is dense in $Soc(H/K)$.
\end{lemma}

\begin{proof}
If $Soc(H/K)=Soc(G/K)$, then it is obvious, hence we may assume that $Soc(H/K)\setminus Soc(G/K)\neq\emptyset$. It suffices to prove that $(H/K)[p]$ is contained in the closure of $Soc(G/K)$ in $H/K$ for each $p\in\mathbb{P}$. Take an arbitrary $g+K\in (H/K)[p]\setminus Soc(G/K)$ for some $g\in H$ and $p\in\mathbb{P}$. We will prove that $\pi(g+U)\cap Soc(G/K)\neq\emptyset$ for any open neighborhood $U$ of $0$ in $H$, where $\pi: H\rightarrow H/K$ is the natural quotient mapping. Indeed, it is obvious that $pg\in K$, then there exists $h\in K$ such that $pg=h$. Since $K$ is divisible, we can find $f\in K$ such that $pg=pf$, thus $g-f\in Soc(H)$. Since $Soc(G)$ is dense in $Soc(H)$, the open set $(g-f+U)\cap Soc(G)\neq\emptyset$, hence there exists $u\in U$ such that $g-f+u\in Soc(G)$, then $g+u\in K+G=G$ and $g+u+K\in Soc(G/K)\cap\pi(g+U)$ since $f\in K$. Therefore, $Soc(G/K)$ is dense in $Soc(H/K)$.
\end{proof}

Let $G$ be a topological group and $K$ a compact normal subgroup of $G$. For convenience, we say that $G$ satisfies $\blacktriangle(G, K)$ if $G\in\mbox{{\bf Trans}}$ implies $G/K\in\mbox{{\bf Trans}}$.

\begin{theorem}\label{th7}
Let $H$ be a topological group containing a dense subgroup $G$ being locally precompact and abelian, and let $K$ be a compact subgroup of $G$. If one of the following conditions holds, then $G$ satisfies $\blacktriangle(G, K)$ if and only if $H$ satisfies $\blacktriangle(H, K)$.

\smallskip
(i) $G$ is connected;

\smallskip
(ii) The subgroup Soc($G$) is dense in Soc($H$), and $K$ is a divisible subgroup.
\end{theorem}

\begin{proof}
Clearly, $G$ and $G/K$ are dense in $H$ and $H/K$ respectively. Moreover, $G/K$ and $H/K$ are all locally precompact. By Lemmas~\ref{l1} and~\ref{lll0}, we easily see that theorem holds.
\end{proof}

\begin{theorem}
Let $G$ be a locally precompact abelian topological group, and let $K$ be a compact subgroup of $G$. If one of the following conditions holds, then $G$ satisfies $\blacktriangle(G, K)$.

\smallskip
(i) $G$ is connected;

\smallskip
(ii) $Soc(G)$ is dense in $G$ and $K$ is a divisible subgroup.
\end{theorem}

\begin{proof}
In \cite{DTY2006}, D. Dikranjan, M. Tkachenko and I. Yaschenko proved that an abelian locally compact topological group $G$ satisfies $\blacktriangle(G, K)$ for any compact subgroup $K$ of $G$. Moreover, since $G$ is a locally precompact abelian topological group, it follows from Theorem~\ref{th6} that $G$ can embed as a dense subgroup of a locally compact abelian group $\widetilde{G}$. Hence we obtain the result by Theorem~\ref{th7}.
\end{proof}

 \maketitle
\section{central topological groups}
In this section, we mainly discuss central topological groups and give some partial answers to Problems~\ref{p1} and~\ref{p2}. Recall that a locally compact topological group $G$ is a {\it central topological group} \cite{GM1967} if $G/Z(G)$ is compact. First, we give a characterization of central topological groups that admit a transversal group topology.

\begin{theorem}\label{th5}
Let $(G, \tau)$ be a non-discrete central topological group. Then $G$ admits a transversal group topology if and only if $G$ satisfies one of the following conditions:

\smallskip
(1) $G$ contains a central subgroup which is topologically isomorphic to $\mathbb{Z}$, where $\mathbb{Z}$ is endowed with a discrete topology.

\smallskip
(2) There exists an open compact subgroup $F$ of $G$ such that $F$ is not open in $(G, \mathcal{M}_{G})$.
\end{theorem}

\begin{proof}
Sufficiency. Assume that $G$ contains a central subgroup which is topologically isomorphic to $\mathbb{Z}$. Then $G$ admits a transversal group topology by \cite[Corollary 3.5]{DTY2005}. Suppose that $G$ contains an open compact subgroup $F$ such that $F$ is not open in $(G, \mathcal{M}_{G})$. Since $$\tau\mid_{F}\leq\mathcal{P}_{F}=\mathcal{P}_{G}\mid_{F}\leq\mathcal{M}_{G}\mid_{F}$$ by \cite[Proposition 3.12]{DTY2005}, it follows that $\tau\nleqslant \mathcal{M}_{G}$ if and only if $F\not\in\mathcal{M}_{G}$. Then $G$ admits a transversal group topology by \cite[Lemma 3.1]{DTY2005}.

Necessity. Assume that $G$ admits a transversal group topology, and assume that $G$ does not contain a central subgroup which is topologically isomorphic to $\mathbb{Z}$, where $\mathbb{Z}$ is endowed with a discrete topology. Since $G$ is a central topological group, it follows from \cite[Theorem 4.4]{GM1967} that $G=V\times O$, where $V$ is a central vector subgroup in $G$ and $O$ is a locally compact group which contains a compact normal subgroup $F$ such that $F$ is open in $O$. We claim that $V=\{0\}$. Suppose not, take any $a=(x, y)\in G\cap(V\times F)$ such that $x\in V\setminus\{0\}$. It easily see that the the infinite cyclic group $<a>$ is a discrete subgroup of $G$ which is topologically isomorphic to the discrete group $\mathbb{Z}$, that is a contradiction. Therefore, $G=O$. Then $G$ contains an open compact subgroup $F$. Assume $F$ is open in $(G, \mathcal{M}_{G})$. Then it follows from the proof of the sufficiency above that $\tau\leq\mathcal{M}_{G}$. Hence $G$ does not admit a transversal group topology by \cite[Lemma 3.1]{DTY2005}, which is a contradiction. Therefore, $F$ is not open in $(G, \mathcal{M}_{G})$.
\end{proof}

\begin{theorem}\label{th2}
Let $G$ be a connected central topological group. Then $G$ admits a transversal group topology if and only if $G$ is not compact.
\end{theorem}

\begin{proof}
If $G$ is compact, then it is obvious that $G$ does not admit a transversal group topology by \cite[Theorem 2.2]{DTY2006}. Hence it suffices to consider that $G$ is not compact. Since $G$ is a connected central topological group, it follows from \cite[Theorem 4.4]{GM1967} that $G=V\times K$, where $V$ is a central vector group in $G$ and $K$ is a compact subgroup. Since $G$ is not compact, $V$ is non-trivial. Then it follows from \cite[Corollary 3.5]{DTY2005} that $G$ admits a transversal group topology.
\end{proof}

Next, we prove that each connected central topological group $G$ satisfies $\blacktriangle(G, K)$ for any compact normal subgroup $K$ of $G$.

\begin{theorem}\label{th3}
If $G$ is a connected central topological group and $K$ is a compact normal subgroup of $G$, then $G$ satisfies $\blacktriangle(G, K)$.
\end{theorem}

\begin{proof}
Assume that $G\in\mbox{{\bf Trans}}$, then $G$ is non-compact by Theorem~\ref{th2}; hence $G/K$ is a non-compact, connected and locally compact topological group. From \cite[(2) of Theorem 2.1]{GM1967}, we see that $G/K$ is a central topological group, hence it admits a transversal group topology by Theorem~\ref{th2}.
\end{proof}

In Theorem~\ref{th3}, $G$ is a connected central topological group. Then it is natural to consider the class of the non-connected locally compact groups. Indeed, we have the following Theorem~\ref{t9}. Recall that a space $X$ is {\it totally disconnected} if each connected subspace is a singleton.

\begin{theorem}\label{t9}
If $(G, \tau)$ is a locally compact totally disconnected group and $K$ is a compact normal subgroup of $G$, then $G$ satisfies $\blacktriangle(G, K)$.
\end{theorem}

\begin{proof}
Since $G$ is a locally compact totally disconnected group, it follows from \cite[Theorem 2]{P1987} that there exists a compact subgroup $L$ of $G$ such that $K\subset L$ and $L$ is open in $(G, \tau)$. Assume that $G\in\mbox{{\bf Trans}}$, then there exists a group topology $\sigma$ on $G$ transversal to $\tau$. We claim that $(G/K)_{\tau}$ and $(G/K)_{\sigma}$ are transversal, where $(G/K)_{\tau}$ and $(G/K)_{\sigma}$ denote the quotient group topologies under the group topologies $\tau$ and $\sigma$ respectively. Indeed, since $L$ is compact in $\tau$, it follows from \cite[Corollary 3.14]{DTY2005} that there exists an open neighborhood $V$ of $e$ in $\sigma$ such that $V\cap L=\{e\}$, hence $VL\cap L=\{e\}$, which implies that $(G/K)_{\tau}$ and $(G/K)_{\sigma}$ are transversal.
\end{proof}

\begin{question}
If $G$ is a locally pseudocompact totally disconnected topological group and $K$ is a compact normal subgroup of $G$, then does $G$ satisfy $\blacktriangle(G, K)$?
\end{question}

Finally, we give a central topological group $G$ such that $G$ admits a transversal group topology; however, it is not connected and does not contain any non-trivial discrete subgroup.

\begin{example}\label{e0}
There exists a central topological group $G$ satisfies the following conditions:

\smallskip
(1) $G$ admits a transversal group topology.

\smallskip
(2) $G$ contains an open compact subgroup $K$; hence $G/K$ is discrete.

\smallskip
(3) $G$ is not connected.

\smallskip
(4) $G$ does not contain any non-trivial discrete subgroup.
\end{example}

\begin{proof}
Let $p\geq 1$ be a prime. Let

\[G=\left\{\left[
\begin{array}{ccc}
    1 & b & a\\
    0 & 1 & c\\
    0 & 0 & 1\\
\end{array}
\right]: a\in\mathbf{Q}_{p}, b, c\in\Delta_{p}\right\},
\]
where $\mathbf{Q}_{p}$ is the locally compact ring of $p$-adic numbers and $\Delta_{p}$ is the ring of $p$-adic integers which is a compact open subring of $\mathbf{Q}_{p}$.

It follows from \cite[Example 5.4]{AK1997} that

\[Z(G)=\left\{\left[
\begin{array}{ccc}
    1 & 0 & a\\
    0 & 1 & 0\\
    0 & 0 & 1\\
\end{array}
\right]: a\in\mathbf{Q}_{p}\right\}.
\]
Put
\[K=\left\{\left[
\begin{array}{ccc}
    1 & b & 0\\
    0 & 1 & c\\
    0 & 0 & 1\\
\end{array}
\right]: a, b, c\in\Delta_{p}\right\}.
\]
Then it easily see that $G=KZ(G)$ and $K$ is an open compact normal subgroup of $G$, see \cite[Example 5.4]{AK1997}. Therefore, $G$ is a central topological group. Obviously, $Z(G)$ is topologically isomorphic to $\mathbf{Q}_{p}$, hence it follows from \cite[Corollary 3.7]{DTY2005} that $Z(G)$ admits a transversal group topology. Since $Z(G)$ is an abelian infinite subgroup, it follows from \cite[Theorem 3.4]{DTY2005} that $G$ admits a transversal group topology. Therefore, $G$ satisfies (1)-(3). Assume that $G$ contains a non-trivial discrete subgroup $H$. Take any $A\in H\setminus\{e\}$. Then there exist $B\in K$ and $C\in Z(G)$ such that $A=BC$. Then $A^{p^{n}}\in H\setminus\{e\}$ for each $n\in\mathbb{N}$. However, since $Z(G)$ is central, each $A^{p^{n}}=B^{p^{n}}C^{p^{n}}$. Then $B^{p^{n}}\rightarrow e$ and $C^{p^{n}}\rightarrow e$ as $n\rightarrow\infty$, hence $A^{p^{n}}\rightarrow e$ as $n\rightarrow\infty$, which is a contradiction.
\end{proof}

{\bf Note} The group $G$ in Example~\ref{e0} is not topologically isomorphic to any group of the form $F\times D$ by (4), where $D$ is a discrete group with $|D|>1$.

  \end{document}